\documentclass[12pt,oneside]{amsart} 
\usepackage{amssymb,amscd}
\usepackage{euscript}

      \theoremstyle{plain}
      \newtheorem{theorem}{Theorem}[section]
      \newtheorem{lemma}[theorem]{Lemma}
      
      \newtheorem{proposition}[theorem]{Proposition}

\numberwithin{equation}{section}

      \makeatletter
      \def\@setcopyright{}
      \def\serieslogo@{}
      \makeatother

\def\A{\EuScript{A}} 
\def\B{\EuScript{B}} 
\def\tB{\tilde{\EuScript{B}} }
\def\n{\mathcal N}
\def\M{\mathcal M}

\def\R{\mathbb R}
\def\Z{\mathbb Z}
\def\N{\mathbb N}
\def\dist{\text{dist}}
\def\Id{\text{Id}}
\def\e{\epsilon}
\def\a{\alpha}

\def\la{\lambda}

\def\tg{\tilde g}

\def\drm{\mbox{Diff}^{\,r}(\M)}
\def\dqm{\mbox{Diff}^{\,q}(\M)}
\def\dtm{\mbox{Diff}^{\,t}(\M)}

\def\donem{\mbox{Diff}^{\,1}(\M)}
\def\dzm{\mbox{Homeo}(\M)}
\def\dkm{\mbox{Diff}^{\,k+1}(\M)}

\begin{document}

\author{Victoria Sadovskaya$^{\ast}$}

\address{Department of Mathematics, The Pennsylvania State University, 
University Park, PA 16802, USA.}
\email{sadovskaya@psu.edu}

\title[Periodic data conjugacy for diffeomorphism cocycles]
{Periodic data conjugacy \\ for diffeomorphism cocycles} 

\thanks{$^{\ast}$ Supported in part by Simons Foundation grant MP-TSM-00002874}


\begin{abstract} 

We consider H\"older continuous cocycles over a hyperbolic dynamical system 
with values in the group of diffeomorphisms of a compact manifold. We focus on establishing
conjugacy between two cocycles from conjugacy of their periodic data. We obtain such a result
under bunching, or domination, assumption on the cocycles and a mild regularity assumption in 
$C^0$ topology on the maps conjugating the periodic data. We also obtain conjugacy to a constant 
cocycle assuming boundedness in $C^1$ or pre-compactness in $C^0$ for the constant cocycle and the
 conjugating maps.

\end{abstract}

\maketitle 

\section{Introduction and statements of the results}

In this paper we consider diffeomorphism-valued cocycles over a hyperbolic dynamical system. 
We focus on the problem of obtaining conjugacy between cocycles from conjugacy of their periodic data.

By a hyperbolic system we mean either a transitive 
Anosov diffeomorphism of a compact connected manifold, or a topologically mixing  
diffeomorphism of a locally maximal 
hyperbolic set, or a mixing subshift of finite type. We study
H\"older continuous cocycles  with values in the group $\dqm$ of $C^q$ diffeomorphisms of a compact 
connected manifold $\M$. H\"older  regularity is needed to obtain results even for scalar cocycles,  and it
allows symbolic systems in the base, which correspond to
random and Markovian sequences of diffeomorphisms.
\vskip.1cm

Let $(X,f)$ be a hyperbolic system 
and let $A:X\to\dqm$. The $\dqm$-valued {\em cocycle} over $f$ generated by $A$ 
is the map 
 $$
 \begin{aligned}
 &\A:\,X \times \Z \,\to \dqm \;\text{ defined  by }\; \A(x,0)=\A_x^0=\Id,\, \text{ and for }n\in \N,\\
&\A(x,n)=\A_x^n = A(f^{n-1} x)\circ \cdots \circ A(x) \quad\text{and}\quad\,
\A(x,-n)=\A_x^{-n}= (\A_{f^{-n} x}^n)^{-1} .
\end{aligned}
$$
We say that a cocycle $\A$  is  {\it H\"older continuous} or {\it bounded} if the property holds for its generator $A(x)=\A_x$.
\vskip.1cm
A function $\Psi:X\to\dtm $ is called a {\it conjugacy} between cocycles $\A$ and $\B$ if
$$
 \A_x=\Psi_{fx}\circ \B_x \circ \Psi_x^{-1} \quad \text{for every $x\in X$}.
$$ 
Clearly, conjugacy of the cocycles implies {\it conjugacy of their periodic data}, which means that  
for every $p=f^n(p)$ there exists a diffeomorphism $\Phi_p$ such that 
$$
\A^n_p=\Phi_p\circ \B_p^n \circ \Phi_p^{-1}.
$$

It is natural to ask whether conjugacy of the periodic data implies conjugacy of the cocycles.
This question for H\"older continuous group-valued cocycles has been extensively studied starting with the work of Liv\v{s}ic \cite{Liv1, Liv2}. In these papers, a positive answer was obtained for cocycles with values in $\R$ and in abelian groups. 
For such groups the question reduces to the case when $\B_x=\Id$, and 
periodic data conjugacy becomes the equality $\A^n_p=\Id$. 
The non-abelian case is much more difficult.
The case when $\B=\Id$ has been studied the most and by now is 
resolved for groups with bi-invariant metrics \cite{Liv2}, $GL(d,\R)$ \cite{K11}, Banach rings \cite{GG}, 
and groups  of diffeomorphism of a compact manifold in low  dimension \cite{KP} and in arbitrary  dimension \cite{AKL}.

For non-abelian groups, the problem for two arbitrary cocycles does 
not reduce to the special case $\B=\Id$. In  general  the answer is negative even for
$GL(2,\R)$-valued cocycles with bounded periodic data conjugacy  \cite{S13}. Therefore some extra assumptions on the cocycles or on the maps $\Phi_p$ are needed. The case of a stronger assumption of equality 
of the periodic data $\A^n_p=\B_p^n$, that is, $\Phi_p=\Id$, is better understood. Cocycles in compact groups 
and cocycles with ``bounded distortion" were studied by Parry \cite{Pa99} and Schmidt  \cite{Sch},
who obtained positive answers for the equality case and some results for the conjugacy case. 
The most general results for $GL(d,\R)$ were established in \cite{S15} 
for fiber bunched cocycles and H\"older continuity of $\Phi_p$ at one point, 
and in \cite{KS26} with no or weaker assumptions of $\Phi_p$ for broad classes of constant cocycles $\B$. 
 \vskip.1cm

We consider  cocycles with values in a group of diffeomorphism of a compact manifold. In this setting,  the only result so far was obtained in \cite{BK} for cocycles 
with {\it equal}\, periodic data satisfying a certain bunching, or domination, assumption. This assumption means that the growth of the cocycle is dominated by the hyperbolicity in the base. It is used in most results for non-commutative cocycles to obtain convergence
of certain iterates of the cocycle, and it yields existence of cocycle holonomies. 
Unlike  \cite{BK}, we study cocycles with {\it conjugate}\, periodic data.  For general bunched cocycles $\A$ and $\B$, we construct a conjugacy under the assumption of H\"older continuity of $\Phi_p$ at one point in $C^0$ topology.  
We also obtain a result for a bunched cocycle $\A$ and a constant cocycle $\B$ under 
certain pre-compactness assumptions on $\B$ and $\Phi_p$ in $C^0$ topology.
Our assumptions on $\Phi_p$ are much weaker than equality of the periodic data.

A particular difficulty in the diffeomorphism case is that the composition operation in $\dqm$ 
is not continuous in $C^q$ topology. This results in loss of regularity, which is two derivatives in  \cite{BK}.
In this paper, we obtain conjugating diffeomorphisms $\Psi_x$ almost as regular as the cocycles.

\vskip.1cm

Now we formulate the  results. We denote the space of homeomorphisms of $\M$ by $\dzm$ and the distance between homeomorphisms by $d_0$.
We  denote the usual $C^r$ norm adapted to the manifold setting by $| . |_{C^r}$, see Section \ref{Crdist}, where a distance $d_{C^r}$ between diffeomorphisms is also defined. In the results below, 
 $0<\la<1$ is so that $\la$ and $\la^{-1}$ are guaranteed contraction and expansion rates in the base system $(X,f)$.

\begin{theorem} \label{conjugate  data}
Let $(X,f)$ be a hyperbolic system and let $\M$ be a compact connected manifold.   
Let $q=k+\gamma$, where $k\in \N=\{1,2, \dots\}$ and $0<\gamma<1$,  and
let $k<r<q$.
\vskip.1cm

Let $\A$ and $\B$ be $\beta$-H\"older continuous  $\dqm$-valued cocycles over $(X,f)$ 
satisfying the following  bunching condition with $\la$ as in \eqref{Anosov def}:
\begin{equation}\label{growth}
\begin{aligned}
 &\text{there exist $K>0$ and $\sigma>1$ such that  $\sigma^{2(1+ r)} \la^{\beta(q-r)}<1$,\, and}\\
 & |\A_x^n|_{C^r} \le K \sigma^n \;\text{ and}\quad |\B_x^n|_{C^r} \le K \sigma^n
\quad\text{for all }x\in X \text{ and }n\in \N.
\end{aligned}
\end{equation}
Suppose that  for each periodic point $p=f^n(p)$ there exists 
$\Phi_p\in \dzm$ such that 
\begin{equation}\label{per conj}
\A^n_p=\Phi_p\circ \B_p^n \circ \Phi_p^{-1},
\end{equation}
and for some periodic point $p_0$ we have   $\Phi_{p_0} \in \dqm$ and for some constant $c$
\begin{equation}\label{Hold at p_0}
d_{C^0}(\Phi_p, \Phi_{p_0}) 
 \le c\, \dist (p,p_0)^\beta \quad\text{for every periodic point $p$ of $f$ close to $p_0$.}
\end{equation}

Then for any $t$ such that  $k<t<r$ there exists a H\"older continuous conjugacy 
$\Psi:X\to \dtm$  between $\A$ and $\B$  with $\Psi_{p_0}= \Phi_{p_0}$.
\end{theorem}

We note that $\Psi_p$  does not necessarily coincide with $\Phi_p$  for $p\ne p_0$. For example, for $\B=\Id$ and any $\A$ of the form $\A_x=\Psi_{fx} \circ\Psi^{-1}_x$,  we can take $\Phi_p=\Id$ at every $p$.
\vskip.1cm

 The H\"older exponent of $\Psi:X\to \dtm$ may be smaller than $\beta.$ In the proof we obtain the exponent $\beta(q-r)(r-t)^2(q-t)$.  In the case of $\Phi_{p_0}=\Id$ we get  $\beta(q-r)(r-t)$.
\vskip.1cm

The bunching condition in the theorem ensures that 
the cocycles $\A$ and $\B$ have stable and unstable holonomies in $\drm$.
The holonomies are defined and their properties are stated in Proposition \ref{holonomies}.
They play a key role in the construction of the conjugacy.
\vskip.2cm

Now we state two results for a {\it constant} cocycle $\B$.

\begin{theorem} \label{Th2}  Let $(X,f)$ be a hyperbolic system with a fixed point, let $\M$ be a compact 
connected manifold,  let $q=k+\gamma$ where $k\in \N$ and $0<\gamma<1$, and let $k<t<q$.

\vskip.1cm 

\begin{itemize}
\item[{\bf (1)}] Let $\B\equiv B$ be a constant $\dqm$-valued cocycle over $(X,f)$ \\
such that the set $\{B^n:\, n\in \Z\}$ is pre-compact in $C^0(\M)$.
\vskip.1cm 

\item[{\bf (2)}] Let $\A$ be a $\beta$-H\"older continuous  $\dqm$-valued cocycle over $(X,f)$  satisfying the following bunching condition with $\la$ as in \eqref{Anosov def}:\\
there exist  $K>0$,\,  $\sigma>1$ and $t<r<q$ such that  $\,\sigma^{2(1+ r)} \la^{\beta(q-r)}<1\,$  and
$$
  |\A_x^n|_{C^r} \le K \sigma^n \quad\text{for all }x\in X \text{ and }n\in \N.
$$
\vskip.1cm

\item[{\bf (3)}]   Suppose that  for each periodic point $p=f^n(p)$ there exists  \\
$\Phi_p\in \dqm$ such that 
$\A^n_p=\Phi_p\circ \B_p^n \circ \Phi_p^{-1}$, and  the set \\
$\{\Phi_p^{-1}:\, p=f^n(p),\;p\in X,\; n\in \N\}$ is pre-compact in $C^0(\M)$.
\vskip.1cm 
\end{itemize}
\vskip.1cm 
Then there exists a H\"older continuous conjugacy $\Psi: X\to \dtm$  between $\A$ and $\B$.
\end{theorem}
\vskip.3cm 

This result  holds with the assumptions $(1, 2, 3)$ replaced by  $(1', 2', 3')$ below.

\begin{theorem} \label{Th1} Let $(X,f)$ be a hyperbolic system with a fixed point, let $\M$ be a compact 
connected manifold,  let $q=k+\gamma$ where $k\in \N$ and $0<\gamma<1$, and let $k<t<q$.

\vskip.1cm 
\begin{itemize}
\item[{\bf (1$'$)}] Let $\B\equiv B$ be a constant $\dqm$-valued cocycle over $(X,f)$ such that 
the set $\{B^n:\, n\in \N\}$ is bounded in $\donem$.
\vskip.1cm 

\item[{\bf (2$'$)}] Let $\A$ be a bounded $\dkm$-valued cocycle over $(X,f)$ that is $\beta$-H\"older continuous 
as a $\dqm$-valued cocycle. 
\vskip.1cm 

\item[{\bf (3$'$)}]  Suppose that  for each periodic point $p=f^n(p)\in X$ there exists 
\\ $\Phi(p)\in \dqm$ such that 
$\A^n_p=\Phi_p\circ \B_p^n \circ \Phi_p^{-1}$, and the set \\
$\{\Phi_p:\, p=f^n(p),\;p\in X,\; n\in \N\}$ is bounded in $\donem$.
\vskip.1cm 
\end{itemize}
Then there exists a H\"older continuous conjugacy $\Psi: X\to \dtm$  between $\A$ and $\B$.

\end{theorem}

We note that  having a set bounded in $\donem$ is equivalent to having  the first derivatives of the maps and their inverses uniformly bounded. 

The assumption in $(2')$ that the cocycle $\A$ is bounded in $\dkm$ is used only  in Proposition \ref{k+1} to obtain bunching from boundedness of the periodic data in $\donem$.
\vskip.1cm

 In Theorems \ref{Th2} and \ref{Th1} we obtain the H\"older exponent $\beta(q-r)(r-t)(q-t)$ for $\Psi$ .

\vskip.2cm

The paper is structured as follows. Section 2 is devoted to the preliminaries: in subsection 2.1 we describe hyperbolic dynamical systems, in 2.2 we define distance in the spaces of homeomorphisms and diffeomorphisms, and in 2.3 we give various estimates for the distances, and in  2.4 we formulate the definition and properties of the holonomies of diffeomorphism-valued cocycles. In Section 3 we prove Theorem \ref{conjugate  data}, and in Section 4 we prove Theorems \ref{Th2} and \ref{Th1}.
\vskip.1cm 
 
The proofs of the three theorems follow the same general scheme.  First we adjust the cocycle $\B$ so that $\A$ and $\B$ have the same value at a fixed point $p_0$. Using holonomies of the cocycles, we define conjugacies $\Psi^s$ and $\Psi^u$ on the stable and unstable leaves of $p_0$. Then we show that $\Psi^s$ and $\Psi^u$ coincide on the homoclinic set of $p_0$,  and thus obtain a conjugacy  $\Psi$ on this  set. The main difference between the proofs of the theorems is in this step. We show that $\Psi$ is H\"older continuous on the homoclinic set, which is dense in $X$, and hence $\Psi$ extends to a H\"older continuous conjugacy between $\A$ and $\B$ on $X$. In Theorem \ref{conjugate  data}, we then deduce the general case of a periodic point $p_0$. The proofs draw on ideas and techniques from \cite{S15,S19,S22,KS26,BK}.

\section{Preliminaries}

\subsection{Hyperbolic dynamical systems}\label{hyperbolic} See \cite{KtH} for details.$\;$\\
Below we describe the three main classes  of these systems.

\vskip.1cm
\noindent{\bf Transitive Anosov diffeomorphisms.} 
A diffeomorphism  $f$ of  compact connected  manifold $X$
 is called {\it Anosov}\, if there exist a splitting 
of the tangent bundle $TX$ into a direct sum of two $Df$-invariant 
continuous subbundles $E^s$ and $E^u$,  a Riemannian 
metric on $X$, and a number $\la$ such that 
\begin{equation}\label{Anosov def}
\|D_x f (v^s)\| < \la < 1 < \la^{-1} <\|D_x f (v^u)\|
\end{equation}
for any $x \in X$ and unit vectors  
$\,v^s\in E^s(x)$ and $\,v^u\in E^u(x)$.
The {\it stable and unstable subbundles} $E^s$ and $E^u$
are tangent to the {\it stable and unstable foliations} 
$W^s$ and $W^u$. 
The {\it local stable manifold} of $x$, $W^{s}_{loc}(x)$, is 
a small ball in $W^s(x)$ centered at $x$  so that 
$$
\dist (f^nx, f^ny)\le \la^n \dist (x,y)  \quad\text{for all }y\in W^s_{loc}(x)\text{ and }n\in \N.
$$
Local unstable manifolds 
are defined similarly. We assume that they are small enough so that  
$\,W^{s}_{loc}(x) \cap W^{u}_{loc}(z)\,$ consists of a single point
for any sufficiently close $x,z\in X$. 

A diffeomorphism is said to be {\it (topologically) transitive} if there is a point $x$ in $X$
with dense orbit. All known examples of Anosov diffeomorphisms have this property. Any transitive Anosov diffeomorphism is {\it topological mixing}, that is, for any open sets $U_1,U_2 \subset X$ there exists $N\in\N$ such that $f^n(U_1)\cap U_2\ne \emptyset\,$ for all $n\ge N$.

\vskip.2cm


\noindent{\bf Mixing diffeomorphisms of locally maximal hyperbolic sets.}
More generally, let $f$ be a diffeomorphism of a manifold $\n$.
A compact $f$-invariant  set $X \subset \n$ is
called {\em hyperbolic} if there  exist a continuous $Df$-invariant splitting 
$T_X \n = E^s\oplus E^u$, and a Riemannian metric on an open set
$U \supset X$ such that \eqref{Anosov def} holds for all $x \in X$.
Local stable and unstable manifolds are defined similarly for any $x \in X$,
 and we denote  their intersections with $X$ by $W^{s}_{loc}(x)$ and  $W^{u}_{loc}(x)$.
The set $X$ is called {\em locally maximal} if 
$X= \bigcap_{n\in \Z} f^{-n }(U)$ for some open set $U\supset X$. 
This property ensures that $W^{s}_{loc}(x) \cap W^{u}_{loc}(z)$ exists in $X$.
We assume that $f|X$ is topologically mixing.
 

\vskip.2cm
\noindent{\bf Mixing subshifts of finite type.}
Let $M$ be $k \times k$ matrix with entries from $\{ 0,1 \} $ such that all 
entries of $M^N$ are positive for some $N$. Let
$$
X= \{ \,x=(x_n) _{n\in \Z}\, : \,\; 1\le x_n\le k \;\text{ and }\;
 M_{x_n,x_{n+1}}=1 \,\text{ for every } n\in \Z \,\}.
$$
\noindent The shift map $f:X\to X\,$ is defined by 
$(f(x))_n=x_{n+1}$.
The system $(X,f)$ is called a {\em  mixing  subshift of finite type}. 
We fix $\la \in (0,1)$ and consider the metric 
$$
d(x,y) = d_\la (x,y)=\la^{n(x,y)},
\;\text{ where }\;n(x,y)=\min\,\{ \,|i|\,: \; x_i \ne y_i  \}.
$$
The following sets play the role of the stable and unstable 
manifolds of $x$
$$
\begin{aligned}
&W^s(x)=\{\,y: \; \text{for some $i_0$}, \;x_i=y_i\, \;\text{for all }\,i\ge i_0\,\}, \\
& W^u(x)=\{\,y:\; \text{for some $i_0$}, \;x_i=y_i\, \;\text{for all }\,i\le i_0\,\},\\
&W^s_{loc}(x)=\{\,y: \; x_i=y_i \;\text{for all }i\ge 0\,\}, \quad 
W^u_{loc}(x)=\{\,y:\; x_i=y_i \;\text{for all }i\le 0\,\}.
\end{aligned}
$$

 \vskip.2cm 
\noindent{\bf Anosov Closing Lemma.}  These systems have the following property.

\begin{lemma} \cite[6.4.15-17]{KtH} \label{closing}\,
Let $(X,f)$ be a hyperbolic dynamical system as above. Then there exist  
constants $c',\, \delta' >0$ such that for any $x \in X$ and $k\in\N$ with 
$\dist (x, f^k (x)) < \delta'$ there exists a periodic point $p \in X$ with 
$f^k p =p$ such that the orbit segments $x, f(x), \dots , f^k (x)$ and 
$p, f(p), \dots , f^k(p)\,$ satisfy
$$
\dist (f^i (x), f^i (p) \le c\,  \dist (x, f^k (x))\;\,\text{ for every $i=0, \dots , k$.}
$$
\end{lemma}

\vskip.3cm

\subsection{Distances on the space $\drm$} \label{Crdist}
(\cite[Section~5]{LW},  \cite[Section 2.2]{S19})\,
We fix a smooth background Riemannian metric and the corresponding distance 
$d_\M$ on $\M$.
We denote the space of homeomorphisms of $\M$ by $\dzm$, and for $g,h\in \dzm$  we set
\begin{equation}\label{d0}
d_0(g,h)= \max_{t\in \M}\, d_\M(g(t),h(t)) \quad\text{and}\quad
d_{C^0}(g,h)=d_0(g,h)+ d_0(g^{-1},h^{-1}).
\end{equation}

Now we consider $r\ge 1$. 
The $C^r$ topology on the group of diffeomorphisms  $\drm$ can
be defined using coordinate patches and the $C^r$ norm in the Euclidean space. 
For any $g\in \drm$,  $r\in \N$, we define its $C^r$ size  as
$$
  |g|_{C^r}= \|g\|_{C^r}+ \|g^{-1}\|_{C^r} , \quad \text{where }\; 
  \|g\|_{C^r}= \max_{t\in \M} d_\M(g(t),t)+ \max_{1 \le i \le r}\max_{t\in \M} \|D_t^ig\|.
$$
Here $D_t^i\,g$ is the derivative of $g$ of order $i$ at $t$, and its norm is defined as
the norm of the corresponding multilinear form from $T_t\M$ to $T_{g(t)}\M$ 
with respect the Riemannian metric. For $r=k+\a$, with $k\in \N$ and  $0<\a<1$, the definition is similar with
$$
 \|g\|_{C^r}=\|g\|_{C^k} + 
 \sup \,\{\, \|D^k_tg -D^k_{t'}g\|\cdot d_\M(t,t')^{-\alpha} : \;t, t'\in \M, \; 0< d_\M(t,t') <\e_0 \}.
$$
Here $\e_0$ is chosen small compared to the injectivity radius of $\M$ so that
the tangent bundle is locally trivialized via parallel transport and the difference makes sense.

\vskip.1cm 
A natural distance $d_{C^r}(g,h)$ on $\drm$,  $r\in \N$,
was defined in \cite{LW}  as the infimum of the lengths of piecewise $C^1$ paths  
in $\drm$ connecting $h$ with $g$  and $h^{-1}$ with $g^{-1}$,
where the length of a path $p_s$ is 
$$
\ell_{C^r} (p_s) =\max_{0 \le i \le r}\max_{t\in \M}\int \|\tfrac d{ds} (D_t^i \, p_s) \|\,ds.
$$
For $r=k+\a$,  one also adds 
the corresponding H\"older term:
$$
\begin{aligned}
&\ell_{C^r} (p_s) =\ell_{C^k} (p_s) 
+ \max_{t\in \M}\int |\tfrac d{ds}  \| (D^k p_s) \|_{\a,t}|\,ds, \quad\text{where}\\
& \|D^k g \|_{\a,t}= \sup \,\{\, \|D^k_{t'}g -D^k_{t}g\|\cdot d_\M(t',t)^{-\alpha} : 
 \,\;t'\in \M, \;\, 0< d_\M(t,t') <\e_0 \}.
 \end{aligned}
$$
For sufficiently $C^r$-close diffeomorphisms, the distance $d_{C^r}(g,h)$ is Lipschitz equivalent 
to $ \|g-h\|_{C^r}+ \|g^{-1}-h^{-1}\|_{C^r}$, where  the difference is understood 
using local trivialization. Specifically,
there exist  constants $\kappa$ and $\delta_0>0$ depending only on $r$ and 
the Riemannian metric so that 
\begin{equation}\label{compare 1}
\begin{aligned}
&\kappa^{-1} d_{C^r}(g,h) \le  \|g-h\|_{C^r}+ \|g^{-1}-h^{-1}\|_{C^r} \le \kappa \, d_{C^r}(g,h), \;\text{ provided that }\\
&\text{either $\;d_{C^r}(g,h)< \delta_0 |g|_{C^r}^{-1}\;$ or 
$\;\|g-h\|_{C^r}+ \|g^{-1}-h^{-1}\|_{C^r}<\delta_0 |g|_{C^r}^{-1}$.}
\end{aligned}
\end{equation}


\subsection{Estimates of norms and distances}$\;$\\
In the arguments we will use the following estimates for distances $d_0$ and $d_{C^r}$.

\begin{lemma}\label{simple}
If the maps below are homeomorphisms of $\M$ and $g$ is differentiable with $\|Dg\|\le M_g$  then 
\begin{equation}\label{d_0 1}
d_0(h_1\circ g, h_2\circ g)=d_0(h_1, h_2)\quad\text{and}\quad
d_0(g\circ h_1, g\circ h_2)\le M_g\, d_0(h_1,h_2)
\end{equation} \vspace{-20 pt}
\begin{equation}\label{d_0 2}
d_0(h_1\circ g\circ h_2,  g) \le  d_0(h_1, \Id)+M_g\, d_0(h_2,\Id)
\end{equation} \vspace{-20 pt}
\begin{equation}\label{d_0 3}
d_0(h_1\circ h_2, \Id)\le  d_0(h_1, \Id)+ d_0(h_2, \Id), 
\end{equation} \vspace{-20 pt}
\begin{equation}\label{d_0 4}
d_0(g\circ h,\, \tilde g\circ \tilde h)\le  d_0 (g, \tilde g) + M_g\, d_0 (h,\tilde h), \quad\text{and}
\end{equation} \vspace{-20 pt}
\begin{equation}\label{d_0 5}
d_0(h_1\circ h_2\circ h_3, \tilde h_1\circ \tilde h_2\circ \tilde h_3,)\le 
  d_0 (h_1, \tilde h_1) + M_{h_1} d_0 (h_2,\tilde h_2) + M_{h_1}M_ {h_2} d_0 (h_3,\tilde h_3),
\end{equation}
provided that $h_1$ and $h_2$ are differentiable with $\|D h_1\|\le M_{h_1}$ and $\|D h_2\|\le M_{h_2}$.
 \end{lemma}
 
 \begin{proof}
 The inequalities  \eqref{d_0 2}  and \eqref{d_0 3} are  obtained, respectively,  using 
 $$
 \begin{aligned}
 & d_0(h_1\circ g\circ h_2,  g)\le d_0(h_1\circ g\circ h_2, g\circ h_2) + d_0(g\circ h_2, g)\quad\text{and} \\
 & d_0(h_1\circ h_2, \Id)\le  d_0(h_1\circ h_2, h_2) +  d_0(h_2, \Id), \\
 \end{aligned} 
$$
and \eqref{d_0 4}, \eqref{d_0 5} follow similarly.
 \end{proof} 
 
\begin{lemma} {\em(Follows directly from \cite[Proposition 5.5]{dlLO})}   \label{dlLO lemma 1}\\
For any $r\ge 1$ there exists a constant  $M_r$ such that for any $h,g \in C^r(\M)$,
$$
  \| h\circ g\|_{C^r}  \le M_r \,  \|h\|_{C^r} (1+\| g\|_{C^r})^r .
  $$
\end{lemma}

\begin{lemma} \cite[Proposition 6.2(iii)]{dlLO}\label{dlLO lemma 2} $\;$\\
 Let $k\in \N$ and $k<t<r<k+1$.
There exist $0<\delta_1<1$ and $ M'>0$ such that if $h_1, h_2 \in C^t(\M)$,\, 
$\| h_1 - h_2\|_{C^t}<\delta_1$,  and $\,g\in C^r(\M)$,\, then
$$
  \| g\circ  h_1 - g\circ h_2\|_{C^t} \le 
  M' \,\|g\|_{C^r} \,\|  h_1 -  h_2\|_{C^t} ^{r-t}.
$$
\end{lemma}

\begin{lemma}\label{comp lemma}
Let $k\in \N$ and $k<t<r<k+1$.
There exists  $M=M(r,t,\M,K)$ such that 
for any $g,\tilde g\in \drm$ and $h_1, h_2\in \dtm$  with $|h_1|_{C^t}, |h_2|_{C^t}\le K$, the following holds. 
\begin{equation} \label{comp 1}
d_{C^t} (g\circ h_1,\, g\circ h_2) \le 
M \left( \|g\|_{C^r}+  (1+\|g^{-1}\|_{C^t})^t\right)\cdot d_{C^t}(h_1, h_2)^{r-t},  \quad\text{and}
\end{equation} 
\vspace{-20pt}
\begin{equation} \label{comp 2}
d_{C^t} (h_1\circ g,\, h_2\circ g) \le 
M \left(\|g^{-1}\|_{C^r}+  (1+\|g\|_{C^t})^t\right)\cdot d_{C^t}(h_1, h_2)^{r-t} \quad\quad\quad 
\end{equation} 
 provided that  $d_{C^t}(h_1, h_2)<\delta |h_1|_{C^t}^{-1}$, where $\delta=\min\{\delta_0, \kappa^{-1}\delta_1 \}$ with $\delta_0, \kappa$ from  \eqref{compare 1} and $\delta_1$ from Lemma \ref{dlLO lemma 2}, and in each case  the right hand side  is less than  
\begin{equation} \label{ass}
 \delta \big(M_t (1+|h_1|_{C^t})^t  (1+|g|_{C^t})^t\big)^{-1}.
 \end{equation} 
Also
 \begin{equation} \label{dist}
\begin{aligned}
&d_{C^t} (g\circ h_1\circ \tg,\, g\circ h_2 \circ \tg)  \\
&\le M \left( \| g\|_{C^r} (1+\| \tg\|_{C^t})^t + \| \tg^{-1}\|_{C^r} (1+ \| g^{-1}\|_{C^t})^t \right)
 \cdot d_{C^t}(h_1, h_2)^{r-t}
\end{aligned}
\end{equation} 
provided that $d_{C^t}(h_1, h_2)\le \delta |h_1|_{C^t}^{-1}$ and  the right hand side 
of \eqref{dist} is less than 
\begin{equation} \label{delta 0}
\delta \left( M_t^{2}\, (1+| h_1|_{C^t})^{t}(   \|g\|_{C^t}  \, (1+\| \tg\|_{C^t})^t +
 \|\tg^{-1}\|_{C^t}   \, (1+\| g^{-1}\|_{C^t})^t )\right)^{-1}.
\end{equation}

\end{lemma}

\begin{proof}  The last part of the lemma was proven in \cite[Lemma 3.6]{S19}. 
For the sake of completeness,  we give a direct argument for \eqref{comp 1} using a similar approach. 
\vskip.1cm 
We use \eqref{compare 1} for  distances between close diffeomorphisms.
 Lemma \ref{dlLO lemma 2} gives an estimate for  $\|g\circ h_1 - g\circ h_2\|_{C^t}$.  
By Lemma \ref{dlLO lemma 1}, we have 
$$
  \|   h_1^{-1} \circ g^{-1}\ -  h_2^{-1} \circ g^{-1}\|_{C^t} = 
  \|  ( h_1^{-1} -   h_2^{-1}) \circ g^{-1}\|_{C^t} \le
 M_t  \|  h_1^{-1} -  h_2^{-1}\|_{C^t}(1+ \|g^{-1}\|_{C^t} )^t.
$$
 Hence 
$$
\begin{aligned}
& \|g\circ h_1 - g\circ h_2\|_{C^t} +  \|   h_1^{-1} \circ g^{-1}\ -  h_2^{-1} \circ g^{-1}\|_{C^t}\le \\
&  M' \,\|g\|_{C^r} \,\|  h_1 -  h_2\|_{C^t} ^{r-t}+ M_t\,\|h_1^{-1}- h_2^{-1}\|_{C^t}(1+\|g^{-1}\|_{C^t})^t \le \\
&M'' \left( \|g\|_{C^r} +(1+\|g^{-1}\|_{C^t})^t\right)\cdot d_{C^t}(h_1, h_2)^{r-t} =: (*).
\end{aligned}
$$
By \eqref{compare 1}, we also have $\kappa^{-1}d_{C^t}(h_1\circ g,\, h_2\circ g)\le(*)$, yielding \eqref{comp 1} with $M=\kappa M''$,  provided that $(*)<\delta_0 |g\circ  h_1|^{-1}_{C^t}$. 
To show the latter  we apply  Lemma \ref{dlLO lemma 1} to estimate 
$$ 
\begin{aligned} 
\|h_1^{-1}\circ  g^{-1}\|_{C^t} &\le M_t\,\|h_1^{-1}\|_{C^t}(1+\|g^{-1}\|_{C^t})^t \le M_t (1+|h_1|_{C^t})^t(1+\|g\|_{C^t})^t,\\
\|g\circ h_1\|_{C^t} &\le M_t (1+\|h_1\|_{C^t})^t \,\| g\|_{C^t}, \;\text{ and so} \\
|h_1\circ  g|_{C^t}& \le M_t (1+|h_1|_{C^t})^t  (1+|g|_{C^t})^t.
\end{aligned}
$$
Thus by the assumption \eqref{ass} we have  $ (*)<\delta |h_1\circ  g|^{-1}_{C^t}\le\delta_0 |h_1\circ  g|^{-1}_{C^t}.$
\end{proof}

\subsection{Holonomies}

In the proofs, a conjugacy $\Psi$ will be  constructed using holonomies of the cocycles $\A$ and $\B$.
We recall the following result on existence and properties of holonomies for diffeomorphism cocycles. 

\begin{proposition}  \cite[Proposition 3.3 and Remark 3.4]{S19}
\label{holonomies} $\;$ \\
Let $k\in \N$ and $k<r<q<k+1$, and let   $\la$ be  as in \eqref{Anosov def}. 
\vskip.1cm 
 Let $\A$ be a is $\beta$-H\"older  $\dqm$-valued cocycle over a hyperbolic  system $(X,f)$. 
Suppose that there exist  $K>0$ and $\sigma>1$ satisfying $\sigma^{2(1+ r)} \la^{\beta(q-r)}<1$  such that 
\begin{equation} \label{C^q bunching}
  |\A_x^n|_{C^q} \le K \sigma^n \;\text{ for all } n\in \N.
\end{equation}

\noindent Then for any $x\in X$ and $y\in W^s(x)$,
the limit 
\begin{equation} \label{holonomy def}
H^s_{x,y}=H^{\A,s}_{x,y}=\lim_{n\to\infty}\, (\A^n_y)^{-1} \circ \A^n_x \quad\text{exists in $\drm$ and satisfies}
\end{equation}

\begin{itemize}
\item[{\bf (H1)}]  $H^s_{x,x}=\Id\,$ and $\,H^s_{y,z} \circ H^s_{x,y}=H^s_{x,z}$,\,\,
which imply $(H^s_{x,y})^{-1}=H^s_{y,x};$
\vskip.1cm
\item[{\bf (H2)}]  $H^s_{x,y}= (\A^n_y)^{-1}\circ H^s_{f^nx ,\,f^ny} \circ \A^n_x\;$ 
for all $n\in \N$;
\vskip.1cm
\item[{\bf (H3)}] There exists a constant $c$ such that for all  $x\in X$   and $y\in W^{s}_{\text{loc}}(x),$
\begin{equation} \label{Anx}
\begin{aligned}
& d_{C^r}(H^s_{\,x,y},\Id) \leq c\,\dist (x,y)^{\beta (q-r)},\quad\text{moreover,}\\
& d_{C^r}  \left( (\A^n_y)^{-1} \circ \A^n_x, \, \Id \right) \le  c\,  \dist(x,y)^{\beta(q-r)} 
  \quad\text{for all }n\in \N.
  \end{aligned}
\end{equation}
\end{itemize}

 \end{proposition}
 
The maps  $H^s_{x,y}$ are called the {\it stable holonomies of $\A$}.
\vskip.1cm 
\noindent The {\it unstable holonomies} are defined similarly: for $y\in W^u(x)$,
$$
H^u_{x,y}=H^{\A,u}_{x,y}=\lim_{n\to-\infty} (\A^n_y)^{-1} \circ \A^n_x \,=\, 
\lim_{n\to\infty}\, (\A^n_{f^{-n}y})\circ (\A^n_{f^{-n}x })^{-1} .
$$
They exist under the same assumptions and satisfy similar properties.
\vskip.2cm


\section{Proof of Theorem \ref{conjugate data}}

\subsection{Reduction to the case of $\A^k_{p_0}=\B^k_{p_0}$ and $\Phi(p_0)=\Id$}$\;$
\vskip.1cm\noindent
Let $k$ be the prime period of $p_0$.
Instead of the cocycle $\B$ we will work with 
$$\,\tB =\Phi_ {p_0}\circ \B \circ \Phi_ {p_0}^{-1},\,\;\text{  so that }\;\tB^k_{p_0} =\A^k_{p_0}.$$ 
Since $\Phi_ {p_0}$ is in $ \dqm$, so is $\tilde B$. Since $\B$ is H\"older continuous with exponent $\beta$, it follows from Lemma \ref{comp lemma} \eqref{dist} that  $\tB$ is H\"older continuous with exponent 
 $\beta(q-r)$.
 
 By Proposition \ref{holonomies}, the original cocycle $\B$ has holonomies in $\drm$ satisfying (H1, H2, H3). The cocycle $\tilde \B$ also has holonomies in $\drm$ given by 
 $$
   H^{\tB,s}_{x,y} =\Phi_ {p_0}\circ H^{\B,s}_{x,y} \circ \Phi_ {p_0}^{-1} \quad\text{ and }\quad
   H^{\tB,u}_{x,y} =\Phi_ {p_0}\circ H^{\B,u}_{x,y} \circ \Phi_ {p_0}^{-1}.
 $$
They satisfy (H1) and (H2), but the property (H3) does not necessarily hold for $H^{\tB,s}$.
However,  by (H3) and \eqref{d_0 1} we have 
 \vskip.2cm
 {\bf (H$\tilde 3$)}   There exists a constant $\tilde c$ such that for all $x\in X$ and $y\in W_{loc}(x)$
  \begin{equation} \label{Anx B0}
d_{C^0}(H^{\tB,s}_{\,x,y},\Id) \leq \tilde c\,\,\dist (x,y)^{\beta (q-r)}.
\end{equation}
Also,  by (H3) and Lemma \ref{comp lemma} \eqref{dist}, 
 \vskip.2cm
 {\bf (H$\bar3$)}  There exists a constant $\bar c$ such that for all $x\in X$ and $y\in W_{loc}(x)$
  \begin{equation} \label{Anx Bt}
 |H^{\tB,s}_{\,x,y}|_{C^r} \le \bar c \quad\text{and} \quad d_{C^t}(H^{\tB,s}_{\,x,y},\Id) \leq \bar c\,\,\dist (x,y)^{\beta (q-r)(r-t)}.
\end{equation}

 The function $\tilde \Phi_p=\Phi_p \circ \Phi_ {p_0}^{-1}$ is a periodic data conjugacy between  $\A$ and $\tilde \B$ with $\tilde \Phi_ {p_0}=\Id$\,  satisfying \eqref{Hold at p_0}.
 
We will construct a H\"older continuous conjugacy $\tilde \Psi:X\to \dtm$ between $\A$ and $\tilde \B$ with the exponent $\beta(q-r)(r-t)^2$. Then $\Psi= \tilde \Psi \circ \Phi_ {p_0}$ will be a conjugacy between $\A$ and $\B$ with values in $\dtm$, H\"older continuous with exponent  $\beta(q-r)(r-t)^2(q-t)$.
\vskip.1cm
 If $\Phi_ {p_0}=\Id$ for the original cocycles, the reduction is not needed and the conjugacy
 $\Psi$ will be H\"older continuous with exponent  $\beta(q-r)(r-t)$.

\vskip.2cm

Now we construct a conjugacy $\tilde \Psi$ between $\A$ and $\tilde \B$. 
To simplify the notations, we will write $\B$ for $\tilde \B$,\, $\Phi$ for $\tilde \Phi$,\, and $\Psi$ for $\tilde \Psi$. Thus,
\vskip.2cm

\noindent {\bf In the rest of the proof of the theorem,}
\vskip.1cm
\noindent $\bullet$\, $\A$ and $\B$ are $\dqm$-valued cocyles over $(X,f)$ \,with\,   $\B^k_{p_0} =\A^k_{p_0} $;
\vskip.1cm
\noindent $\bullet$\, $\A$ and $\B$ have stable and unstable holonomies in $\drm$, 
\vskip.1cm
the holonomies $H^{\A,s/u}$ for $\A$ satisfying (H1,\,H2,\,H3),
\vskip.1cm
the holonomies $H^{\B,s/u}$ for $\B$ satisfying (H1,\,H2,\,H$\tilde 3$,\,H$\bar 3$);
\vskip.1cm
\noindent $\bullet$\,  $\Phi_p$ is a periodic data conjugacy between  $\A$ and $B$ with $\Phi_ {p_0}=\Id$\,  satisfying 
\begin{equation}\label{Hold at p_0 _0}
d_{C^0}(\Phi_p, \Id) 
 \le c\, \dist (p,p_0)^\beta \quad\text{for every periodic point $p$ of $f$ close to $p_0$.}
\end{equation}
\vskip.2cm 


First we assume that $p_0$ is a {\bf fixed point} of $f$  and construct a conjugacy between cocycles $\A$ and $\B$ with $\A_{p_0}=\B_{p_0}$ . We remove this assumption in Section \ref{remove}.
\vskip.2cm

\subsection{Construction of a conjugacy $\Psi$ on the homoclinic set of a fixed point.} \label{Th 1 constr}
We construct conjugacies between $\A$ and $\B$
 along the stable and unstable manifolds of $p_0$ and show that they coincide at every homoclinic point of $p_0$. We set
\begin{equation}\label{Psi su conj}
\begin{aligned}
 &\Psi^s_x=H^{\A,s}_{p_0,\,x}\circ H^{\B,s}_{x,\,p_0} \quad\text{for }x\in W^s(p_0),\\
 &\Psi^u_x=H^{\A,u}_{p_0,\,x}\circ H^{\B,u}_{x,\,p_0} \quad\text{for }x\in W^u(p_0).
 \end{aligned}
\end{equation}
 The function $\Psi^s$ is a conjugacy between $\A$ and $\B$ along the leaf $W^s(p_0)$. Indeed,
 If $x\in W^s(p_0)$, then using property (H2) of the holonomies with $n=1$ we obtain
 $$
 \begin{aligned}
&\Psi^s_{f(x)}\circ \B_x\circ \Psi^s(x)^{-1}\,=\, 
H^{\A,s}_{f(p_0),\,f(x)}\circ H^{\B,s}_{f(x),\,f(p_0)} \circ  \B_x\circ (H^{\B,s}_{x,p_0})^{-1} \circ (H^{\A,u}_{p_0,x})^{-1}\,=\\
&H^{\A,s}_{f(p_0),\,f(x)}\circ \B_{p_0}\circ (H^{\A,u}_{p_0,x})^{-1} \,=\, H^{\A,s}_{f(p_0),\,f(x)}\circ \A_{p_0}\circ (H^{\A,u}_{p_0,x})^{-1} \,=\,\A_x.
\end{aligned}
 $$
Similarly,  $\Psi^u$ is a conjugacy between $\A$ and $\B$ along the leaf $W^u(p_0)$.
\vskip.2cm

Let $x\in W^s(p_0)\cap W^u(p_0)$. We will show that $\Psi^s_x=\Psi^u_x$, equivalently, 
$$
H^{\B,u}_{p_0,x} \circ H^{\A,u}_{x,p_0}\circ H^{\A,s}_{p_0,\,x}\circ H^{\B,s}_{x,\,p_0}=\Id.
$$
Using the definition of the holonomies and the assumption that $\A_{p_0}=\B_{p_0}$, we obtain
$$
H^{\A,s}_{p_0,\,x} \circ H^{\B,s}_{x,\,p_0} =
\underset{n\to\infty}{\lim}
 \left(  (\A^n_x)^{-1} \circ \A^n_{p_0} \circ (\B^n_{p_0})^{-1} \circ  \B^n_x \right)=
 \underset{n\to\infty}{\lim}
 \left(  (\A^n_x)^{-1}\circ  \B^n_x \right)
$$
and  
$$
H^{\B,u}_{p_0,\,x} \circ H^{\A,u}_{x,\,p_0} =
 \underset{n\to\infty}{\lim}
 \left(   \B^n_{f^{-n}x} \circ  (\A_{f^{-n}x}^n)^{-1} \right).
$$
Thus,
$$
\begin{aligned}
H^{\B,u}_{p_0,\,x} \circ H^{\A,u}_{x,\,p_0} \circ H^{\A,s}_{p_0,\,x} \circ H^{\B,s}_{x,\,p_0}
 & =  \underset{n\to\infty}{\lim}
 \left(  \B^n_{f^{-n}x} \circ  (\A_{f^{-n}x}^n)^{-1} \circ  (\A^n_x)^{-1}\circ  \B^n_x  \right)=\\
&  =  \underset{n\to\infty}{\lim}
 \left(  \B^n_{f^{-n}x}\circ  (\A^{2n}_{f^{-n}x})^{-1}\circ  \B^n_x  \right).
  \end{aligned}
 $$
We will show that  $d_0(\B^n_{f^{-n}x}\circ  (\A^{2n}_{f^{-n}x})^{-1}\circ  \B^n_x,\, \Id)\to 0$. 
\vskip.1cm

We  consider an orbit segment $\{f^{-n}x, \dots,x, \dots, f^{n}(x)\}$.
Since
$$
\dist (f^{-n}(x), f^{n}(x))\le \dist (f^{-n}(x), {p_0})+\dist (f^{n}(x),{p_0})\le K_x \la^{n} 
\quad \text{for all $n\in \N$},
$$
for all sufficiently large $n$ we have
$$\delta_n= \dist (f^{-n}(x),f^{n}(x))\le K_x \la^{n} <\delta',
$$ 
where $\delta'$ is as in the Anosov Closing Lemma  \ref{closing}.
Hence there exists a periodic point $p_n=f^{2n}(p_n)$ such that 
such that 
$$
\dist (f^i (x), f^i (p_n)) \le c'  \delta_n\;\,\text{ for every $i=-n, \dots , n$.}
$$
To simplify the notations we will write $p$ for $p_n$. 
\vskip.1cm 
We will prove that 
$$
\begin{aligned}
& d_0(\B^n_{f^{-n}x}\circ (\A^{2n}_{f^{-n}x})^{-1}\circ  \B^n_x,\,\,\,\B^n_{f^{-n}p}\circ (\A^{2n}_{f^{-n}p})^{-1}\circ  \B^n_p)\to 0 \quad\text{and}\quad\\
& d_0(\B^n_{f^{-n}p}\circ  (\A^{2n}_{f^{-n}p})^{-1}\circ  \B^n_p,\, \Id)\to 0.
\end{aligned}
$$
We write  $\B^n_{f^{-n}x}$,  $\B^n_x$, and   $(\A^{2n}_{f^{-n}x})^{-1}$ in terms of the values at $p$.  
Let 
$$
y=y_n=W^s_{loc}(p)\cap W^u_{loc}(x).
$$ 
Then for each $i=-n,\dots, n$, 
$$
\dist (f^i (x), f^i (y)) \le c_1\delta_n\; \text{ and }\;\dist (f^i (y), f^i (p)) \le c_1\delta_n.
$$
By property (H2) of the holonomies  we have
$$
\B^{n}_x\,=\,  H^{\B,u}_{f^{n}y,\,f^{n}x}\circ \B^n_y \circ H^{\B,u}_{x,\,y}\,=\,
  H^{\B,u}_{f^{n}y,\,f^{n}x}\circ H^{\B,s}_{f^{n}p,\,f^{n}y}\circ 
\B^n_p \circ H^{\B,s}_{y,\,p} \circ H^{\B,u}_{x,\,y}.
$$
By property (H3), for each of the holonomies above we have 
$$
d_{C^0}(H,\Id) \le \tilde c(c_1\delta_n)^{\beta(q-r)}=c_2\la^{n\beta(q-r)},
$$
and by \eqref{d_0 3} it follows that 
$$
\B^{n}_x=H_1\circ \B^n_p \circ H_2, \;\text{ where }\; d_0(H_1,\Id)\le  2c_2\,\la^{n\beta(q-r)}
\;\text{ and }\; d_0(H_2,\Id)\le 2c_2\,\la^{n\beta(q-r)}.
$$
By the bunching assumption, 
$$
\|D\B^n_p\| \le K \sigma^n \quad\text{where }\; \sigma^{2(1+r)} \la^{\beta(q-r)}<1.
$$
Using \eqref{d_0 2} we estimate
$$
\begin{aligned}
d_0(\B^{n}_x,\, \B^n_p)&=\,d_0(H_1\circ \B^n_p \circ H_2,\, \B^n_p) \,\le\,  d_0(H_1, \Id)+ \|D\B^n_p\|\cdot  d_0(H_2, \Id)\\
& \le  2c_2\,\la^{n\beta(q-r)} + K \sigma^n \cdot 2c_2\,\la^{n\beta(q-r)} \le 
\, c_3\, \sigma^n \la^{n\beta(q-r)}.
\end{aligned}
$$
Similarly, 
$$
d_0(\B^n_{f^{-n}x},\,\B^n_{f^{-n}p} )\le c_3 \,  \sigma^n  \la^{n\beta(q-r)}  \;\;\text{ and }\;\;
d_0((\A^{2n}_{f^{-n}x})^{-1},\, (\A^{2n}_{f^{-n}p})^{-1}) \le c_3 \,  \sigma^{2n}  \la^{n\beta(q-r)}.
$$
Using \eqref{d_0 5} we deduce
$$
\begin{aligned}
&d_0(\B^n_{f^{-n}x}\circ  (\A^{2n}_{f^{-n}x})^{-1}\circ  \B^n_x,\,\, \B^n_{f^{-n}p}\circ  (\A^{2n}_{f^{-n}p})^{-1}\circ  \B^n_p) \le\\
& d_0(\B^n_{f^{-n}x}, \B^n_{f^{-n}p}) + \|D\B^n_p\| \, d_0((\A^{2n}_{f^{-n}x})^{-1}, (\A^{2n}_{f^{-n}p})^{-1}) \\
&\hskip3cm + \|D\B^n_p\|\, \|D(\A^{2n}_{f^{-n}p})^{-1}\|\, d_0(\B^{n}_x, \B^n_p) \le\\
&c_3\, \sigma^n  \la^{n\beta(q-r)} + K \sigma^n c_3\, \sigma^{2 n}  \la^{n\beta(q-r)} 
+K \sigma^n K \sigma^{2n} c_3\, \sigma^n \la^{n\beta(q-r)}
\le \,c_4\,  \sigma^{3 n}   \la^{n\beta(q-r)}.
\end{aligned}
$$
Since $r>1$ we have  $\sigma^{3}\la^{\beta(q-r)}<\sigma^{2(1+ r)} \la^{\beta(q-r)}<1$,\, and hence  $\,\sigma^{3n}   \la^{n\beta(q-r)}\to 0$\, as $n\to\infty$.

\vskip.2cm 
Next we consider $\B^n_{f^{-n}p} \circ  (\A_{f^{-n}p}^{2n})^{-1} \circ  \B^n_p$.\, We observe that 
$$
 (\A_{f^{-n}p}^{2n})^{-1} \,= \, 
 \Phi_{f^{-n}p}\circ (\B_{f^{-n}p}^{2n})^{-1} \circ \Phi_{f^{-n}p}^{-1}, 
$$
where 
$$
d_{C^0}\left(\Phi_{f^{-n}p}, \Id) = d_{C^0}(\Phi_{f^{-n}p}, \Phi_{p_0}\right)  
 \le c\, \dist (f^{-n}p,p_0)^{\beta} \le c \la^{n\beta} .
$$
Then using \eqref{d_0 2} we obtain
$$
\begin{aligned}
&d_0((\A_{f^{-n}p}^{2n})^{-1},\, (\B_{f^{-n}p}^{2n})^{-1}) = d_0( \Phi_{f^{-n}p}\circ (\B_{f^{-n}p}^{2n})^{-1} \circ \Phi_{f^{-n}p}^{-1},\, (\B_{f^{-n}p}^{2n})^{-1}) \le\\
& d_0( \Phi_{f^{-n}p}, \Id)+  \|D(\B_{f^{-n}p}^{2n})^{-1} \|\,d_0(\Phi_{f^{-n}p}^{-1},\Id)\le 
c \la^{n\beta} +  K  \sigma^{2 n} c \la^{n\beta} \le  c_5 \, \sigma^{2 n} \la^{n\beta}.
\end{aligned}
$$
Finally, using \eqref{d_0 1} we estimate
$$
\begin{aligned}
& d_0( \B^n_{f^{-n}p} \circ  (\A_{f^{-n}p}^{2n})^{-1} \circ  \B^n_p, \,\, \Id)=\\
& d_0( \B^n_{f^{-n}p} \circ  (\A_{f^{-n}p}^{2n})^{-1} \circ  \B^n_p,\,\,\B^n_{f^{-n}p} \circ  (\B_{f^{-n}p}^{2n})^{-1} \circ  \B^n_p)=\\
& d_0( \B^n_{f^{-n}p} \circ  (\A_{f^{-n}p}^{2n})^{-1},\,\,\B^n_{f^{-n}p} \circ  (\B_{f^{-n}p}^{2n})^{-1})\le \\
 &  \|D \B^n_{f^{-n}p}\|\, d_0((\A_{f^{-n}p}^{2n})^{-1},\, (\B_{f^{-n}p}^{2n})^{-1}) 
 \le\,c_6 \, \sigma^{3 n}   \la^{n\beta} \to 0 \quad\text{as }n\to\infty. 
\end{aligned}
$$
Therefore 
$$d_0(\B^n_{f^{-n}x}\circ  (\A^{2n}_{f^{-n}x})^{-1}\circ  \B^n_x,\, \Id)\to 0  \quad\text{as }n\to\infty,
$$
and we conclude that 
 $\Psi^s_x=\Psi^u_x.$
\vskip.2cm 


\subsection{Extending the conjugacy $\Psi$ to $X$.}$\;$\\
We denote the homoclinic set of the fixed point $p_0$ by $S$, that is, 
$$S=W^s(p_0)\cap W^u(p_0).$$
 For each $x\in S$, we define 
$$
\Psi_x=\Psi^s_x=\Psi^u_x= H^{\A,s}_{p_0,\,x}\circ H^{\B,s}_{x,\,p_0} =H^{\A,u}_{p_0,\,x}\circ H^{\B,u}_{x,\,p_0}.
$$

The homoclinic set $S$  is invariant under $f$ and  dense in $X$ \cite{Bo}. 
The function $\Psi:S\to \drm$ conjugates the cocycles on $S$.
We will show that $\Psi$ is H\"older continuous as function from $S$ to $\dtm$. Then it  extends to a H\"older continuous function from $X$ to $\dtm$
that conjugates the cocycles. 

First we establish boundedness of $\Psi$ on $S$ in $C^r$. We take a sufficiently small $\e>0$,  let $S'$ be a finite $\e$-dense subset of $S$, and let 
$$
 M_{S'}=\max \{\|\Psi_w\|_{C^r} : \, w\in S'\}.
$$ 
Let $x\in S$. Then there is $w\in S'$ with $\dist(x,w)<\e$, and we let $z=W^u_{loc}(x)\cap W^s_{loc}(w)$.
Then $z\in S$ and
$$
   \Psi_x= H^{\A,u}_{z,\,x}\circ  \Psi_z \circ H^{\B,u}_{x,\,z} = 
   H^{\A,u}_{z,\,x}\circ H^{\A,s}_{w,z}\circ  \Psi_w \circ H^{\B,s}_{w,\,z} \circ H^{\B,u}_{x,\,z}.
$$
By property (H3) and (H$\bar 3$) of the holomomies, there exists a constant $M_H$ such that 
$$
 \|H^{\A,u/s}_{x,\,y}\|_{C^r}\le M_H\;\text{ and }\; \|H^{\B,u/s}_{x,\,y}\|_{C^r}\le M_H
 \;\text{ whenever $y\in W^{s/u}_{loc}(x).$}
$$
Applying Lemma \ref{dlLO lemma 1}  four times we obtain
$$
\begin{aligned}
&\|\Psi_x\|_{C^r} \,=\, \|H^{\A,u}_{z,\,x}\circ H^{\A,s}_{w,z}\circ  \Psi_w \circ H^{\B,s}_{w,\,z} \circ H^{\B,u}_{x,\,z}\|_{C^r}\le \\
& M_r \,\|H^{\A,u}_{z,\,x}\circ H^{\A,s}_{w,z}\circ  \Psi_w \circ H^{\B,s}_{w,\,z}\|_{C^r}\cdot 
(1+\|H^{\B,u}_{x,\,z}\|_{C_r})^r \le  \\
&M_r^2 \,\|H^{\A,u}_{z,\,x}\circ H^{\A,s}_{w,z}\circ  \Psi_w \|_{C_r}\cdot(1+\| H^{\B,s}_{w,\,z}\|_{C^r})^r\cdot 
(1+\|H^{\B,u}_{x,\,z}\|_{C_r})^r \le \dots \le \\
& M_r^4 \cdot M_H\cdot (1+M_{s'})^r\cdot  (1+M_H)^{3r} =:M_S.
\end{aligned} 
$$
Similarly, $\|\Psi_x^{-1}\|_{C^r} \le M_S.$
\vskip.2cm

Let $x$ and $y$ be two sufficiently close points in $S$. 
The point $z=W^u_{loc}(x)\cap W^s_{loc}(y)$ is also in $S$, and 
$\max \{ \dist(x,z),\dist(y,z)\} \le c_7 \,\dist(x,y).$

Denoting $H^{\A,u}_{z,\,x}\circ  \Psi_z$ by $g$ and  applying Lemma \ref{comp lemma} we obtain
$$
\begin{aligned}
& d_{C^t}(\Psi_x, \Psi_z)=d_{C^t} (H^{\A,u}_{z,\,x}\circ  \Psi_z \circ H^{\B,u}_{x,\,z},\, \Psi_z)\le \\
& d_{C^t} (H^{\A,u}_{z,\,x}\circ  \Psi_z \circ H^{\B,u}_{x,\,z},\,  H^{\A,u}_{z,\,x}\circ \Psi_z\circ\Id) +
d_{C^t} (H^{\A,u}_{z,\,x}\circ \Psi_z, \,\Id\circ \Psi_z))\le \\
& M\left(\|g\|_{C^r}+(1+\|g^{-1}\|_{C^r})^r\right)\cdot d_{C_t}(H^{\B,u}_{x,\,z}, \Id)^{r-t}\,+\\
& \hskip3cm  +M(\left(\|\Psi_z\|_{C^r}+(1+\|\Psi_z^{-1}\|_{C^r}\right)^r)\cdot d_{C_t}(H^{\A,u}_{z,\,x}, \Id)^{r-t}.
\end{aligned}
$$
Since $\|H^{\A,u}_{z,\,x}\|_{C^r}\le M_H$ and  $\|\Psi_x\|_{C^r}\le M_S$, the norm  $\|g\|_{C^r}$ is bounded by Lemma \ref{dlLO lemma 1}. The ``provided that" condition in Lemma \ref{comp lemma} holds since  the quantity \eqref{ass} is uniformly bounded away from 0, and by taking points $x$ and $y$, and hence $x$ and $z$, sufficiently close we can make  $d_{C_t}(H^{\A,u}_{z,\,x}, \Id)$ and $d_{C_t}(H^{\B,u}_{z,\,x}, \Id)$ as close to 0 as needed.

Using  properties (H3) and (H$\bar 3$) of the holomomies of $\A$ and $\B$ respectively we obtain
$$  
 d_{C^t}(\Psi_x, \Psi_z)\le M'd(x,z)^{\beta(q-r)(r-t)^2} \le M''d(x,y)^{\beta(q-r)(r-t)^2}.
$$

We conclude that $\Psi:S\to \dtm$ is H\"older continuous on $S$ with the exponent $\beta(q-r)(r-t)^2$. Hence extends to a H\"older continuous function $\Psi:X\to \dtm$ with the same exponent,  and the extension conjugates the cocycles.
\vskip.3cm


\subsection{Removing the fixed point assumption}\label{remove}$\;$\\
Now suppose that $p_0$ is a periodic point of $f$ with minimal period $k>1$, and so $p_0$ is fixed by $f^k$.

We consider the iterated cocycles $\A^k$ and $\B^k$  over $f^k$ with  $\A^k_{p_0}=\B^k_{p_0}$ and $\Phi_{p_0}=\Id$. The arguments in Sections 3.2 and 3.3  give a H\"older continuous conjugacy $\bar \Psi:X\to \dtm$ between 
$\A^k$ and $\B^k$ with $\bar \Psi_{p_0}=\Id$.
We will show that $\bar \Psi$ is also a conjugacy between $\A$ and $\B$
following the approach of \cite{BK, Sch}. 
\vskip.1cm

We consider a point $x\in W^s(p_0) \cap W^u(f^{k-1}(p_0))$, so that $f(x)\in  W^u(p_0)$. 
We will show that
\begin{equation}\label{Ax}
\bar\Psi_x= \A_x ^{-1} \circ \bar\Psi_{f(x)}\circ \B_x , \quad \text{that is,}\quad 
  \A_x=\bar\Psi_{fx}\circ \B_x\circ \bar\Psi_x^{-1}.
\end{equation}
Since such points $x$ are dense in $X$ and all the functions are continuous,
it follows that the equation holds for all $x\in \M$.

Now we prove \eqref{Ax}. Using equations \eqref{Psi su conj} we obtain 
$$
\begin{aligned}
\bar\Psi_x &= \bar\Psi^s_x =H^{\A^k,s}_{p,\,x}\circ H^{\B^k,s}_{x,\,p} = 
\\ & =\underset{n\to\infty}{\lim} \,(\A^{nk}_x)^{-1} \circ \A^{nk}_p \circ  (\B^{nk}_p)^{-1} \circ \B^{nk}_x
=\underset{n\to\infty}{\lim} \, (\A^{nk}_x)^{-1} \circ  \B^{nk}_x, \quad\text{and}\\
\bar\Psi_{fx} &= \bar\Psi^u_{fx}=H^{\A^k,u}_{p,\,fx}\circ H^{\B^k,u}_{fx,\,p} =
\underset{n\to\infty}{\lim} \,\A^{nk}_{f^{-nk}fx} \circ (\B^{nk}_{f^{-nk}fx})^{-1}  \\
&= \underset{n\to\infty}{\lim} \,\A^{nk}_{f^{-nk+1}x} \circ (\B^{nk}_{f^{-nk+1}x})^{-1}.
 \end{aligned}
$$
Then  
$$
   \A_x ^{-1} \circ \bar\Psi_{fx}\circ \B_x = 
   \underset{n\to\infty}{\lim} 
   \, \A^{nk-1}_{f^{-nk+1}x} \circ (\B^{nk-1}_{f^{-nk+1}x})^{-1}.
    $$
So we need to show  that 
$$
\underset{n\to\infty}{\lim}\, (\A^{nk}_x)^{-1} \circ  \B^{nk}_x  = \underset{n\to\infty}{\lim} 
   \, \A^{nk-1}_{f^{-nk+1}x} \circ (\B^{nk-1}_{f^{-nk+1}x})^{-1}. 
$$
Since both limits exist in $\drm$, so does 
   $$
\underset{n\to\infty}{\lim}\,  \B^{nk-1}_{f^{-nk+1}x} \circ (\A^{nk-1}_{f^{-nk+1}x})^{-1} \circ (\A^{nk}_x)^{-1} \circ  \B^{nk}_x   = \underset{n\to\infty}{\lim}\, \B^{nk-1}_{f^{-nk+1}x} \circ (\A^{2nk-1}_{f^{-nk+1}x})^{-1} \circ  \B^{nk}_x, 
$$
and it suffices to prove that this limit equals the identity.
We show that 
$$
 d_0\big( \B^{nk-1}_{f^{-nk+1}x} \circ (\A^{2nk-1}_{f^{-nk+1}x})^{-1} \circ  \B^{nk}_x, \,\Id\big)\to 0 \;\text{ as }n\to\infty.
$$

We recall that  $x\in W^s(p_0)$ and  $f(x)\in W^u(p_0)$. Hence  $ f^{-nk+1}(x)=f^{-nk}(f(x))$ is in  $W^u(p_0)$ for every $n\in\N$. Therefore by \eqref{Anosov def} 
there is a constant $K_x$ such that for all $n\in \N$,
$$ 
\dist (f^{nk}x,\, p_0)< \text{\tiny $\frac12$} K_x\la^{nk} \quad\text{and}\quad
\dist (f^{-nk+1}x, \,p_0)< \text{\tiny $\frac12$} K_x\la^{nk}.
$$
Let $\delta'$ be as in Anosov Closing Lemma \ref{closing}. For all sufficiently large $n$, we have 
$$
\delta_n=\dist (f^{nk}x,f^{-nk+1}x)<K_x \la^{nk} <\delta',
$$
and we can apply the lemma to the 
the orbit segment 
$\{f^i x:  i=-nk+1, \dots , nk\}.$
Thus there exists a periodic point $p_n=f^{2nk-1}(p_n)\,$ such that 
$$
 \dist (f^i (x),\, f^i (p_n))\le c' \delta_n <c'K_x\la^{nk}\quad \text{for }i=-nk+1,\, \dots, nk.
$$
We note that  $\dist(p_n, p_0)\le c''\la^{nk}$.
\vskip.05cm
The rest of the argument is almost identical to the one in Section \ref{Th 1 constr}, where we showed that  
$$
d_0\big( \B^{n}_{f^{-n}x} \circ (\A^{2n}_{f^{-n}x})^{-1} \circ  \B^{n}_x, \,\Id\big)\to 0 \;\text{ as }n\to\infty.
$$
\vskip.1cm 
Therefore $\bar\Psi:X\to \dtm$ is also a conjugacy between $\A$ and $\B$ with $\bar\Psi (p_0)=\Id.$
This completes the proof of Theorem \ref{conjugate  data}.


\section{Proofs of Theorems \ref{Th2} and \ref{Th1}}

The stable and unstable holonomies of the constant cocycle $\B$ are trivial, that is, equal the identity map on $\M$. In Theorem \ref{Th2}, condition \eqref{C^q bunching} is satisfied by the assumption on  $\A$ in $(2)$. 
In Theorem \ref{Th1}, we use the following proposition on cocycles with bounded periodic data.

\begin{proposition} \label{k+1}
{\em (From  Theorem 1.2(ii) in \cite{S19} and beginning of its proof.)}$\;$\label{bdd data}
\vskip.1cm 

 \noindent  Let $(X,f)$ be a hyperbolic system, let $\M$ be a compact 
connected manifold,
 let $k\in \N$,  and let  $\A$ be a bounded $\dkm$-valued cocycle over $f$.
 \vskip.1cm 
 
\noindent  If the periodic data set $\{ \A_p^k:\; p=f^kp,\; p\in X,\; k\in \N \}$ is  bounded in $\donem$, then 
\vskip.1cm 
\noindent {\bf (i)} The set $\{ \A_x^n: \; x\in X, \; k\in \Z\}$ is also bounded in $\donem$,
\vskip.1cm 
\noindent  {\bf (ii)} For each $\sigma>1$ there exists $K>0$ such that for all $x\in X$,

\hskip.5cm $|\A_x^n|_{C^{k+1}}\le K \sigma^{|n|}$ for all $n\in \Z$.

\end{proposition}

By the assumption (1$'$), the cocycle $\B$ in the Theorem \ref{Th1} has a periodic data set bounded in $\donem$. Since periodic data conjugacy $\Phi$ is also bounded in $\donem$ by (2$'$), it follows that the periodic data of the cocycle $\A$
is bounded in $\donem$.  Therefore, $\A$ satisfies the conclusions (i) and (ii) of Proposition \ref{bdd data}.  We take $r=\frac12(t+q)$. Then $\A$  has stable holonomies as in Proposition \ref{holonomies}, as well as the unstable holonomies. 
\vskip.1cm 

 Thus in Theorems \ref{Th1} and \ref{Th2}, the cocycle $\A$ has stable and unstable holonomies in $\drm$ as in Proposition \ref{holonomies}. Since the holonomies of $\B$ are trivial, we will write just $H^s_{x,y}$ and $H^u_{x,y}$ for the holonomies  $H^{\A,s}_{x,y}$ and $H^{\A,u}_{x,y}$ of $\A$.

 
 \subsection{Construction of a conjugacy $\Psi$ on the homoclinic set of a fixed point.}

Let ${p_0}$ be a fixed point of $f$.
 As in the proof of Theorem \ref{conjugate data} we conjugate the cocycle $\B$ by $\Phi(p_0)$ and obtain another constant cocycle. So we assume that 
$$
\A_{p_0}=\B_{p_0}=B \quad\text{and}\quad \Phi_ {p_0}=\Id.
$$
We will construct a H\"older continuous conjugacy $\Psi: X\to \dtm$ between $\A$ and the adjusted cocycle with the exponent $\beta(q-r)(r-t)$. Then $\Psi= \tilde \Psi \circ \Phi_ {p_0}$ will be a conjugacy between $\A$ and and the original $\B$ with values in $\dtm$, H\"older continuous with exponent  $\beta(q-r)(r-t)(q-t)$.

\vskip.2cm 
First we define conjugacies $\Psi^s$ and $\Psi^u$ on the stable and unstable leaves of the fixed point $q$. 
We set 
$$
\begin{aligned}
&\Psi^s_{p_0}=\Psi^u_{p_0}=\Phi_{p_0}=\Id,\\
 &\Psi^s_x=H^{s}_{{p_0},x}\quad\text{for }x\in W^s({p_0}),\\
 &\Psi^u_x= H^{u}_{{p_0},x}\quad\text{for }x\in W^u({p_0}).
 \end{aligned}
$$
By the property (H2) of the holonomies we have $\A_x= H^{s}_{fx,fy}\circ \A_y\circ H^{s}_{x,y}$. Using this 
as well as ${p_0}=f({p_0})$ and $\A_{p_0}=B$, we obtain 
$$
\Psi^{s}_{f x} \circ B \circ (\Psi^{s}_x)^{-1}=H^{s/u}_{{p_0},fx} \circ B \circ (H^{s/u}_{f({p_0}),\,x})^{-1}=
H^{s}_{{p_0},f(x)} \circ \A_{p_0} \circ H^{s}_{x,\,{p_0}}=\A_x.
 $$
So $\Psi^s$ is a conjugacy between $\A_x$ and $B$ on $W^s({p_0})$, and similarly $\Psi^u$ is  a conjugacy on $W^u({p_0})$.

Let $x$ be a {\it homoclinic point} for $q$, that is, $x\in W^s({p_0})\cap W^u({p_0})$. We will show that 
$$
\Psi^s_x=\Psi^u_x, \quad\text{that is,} \quad 
(\Psi^s_x)^{-1}\circ \Psi^u_x=H^s_{x,{p_0}}\circ H^u_{x,{p_0}}=\Id.
$$
\vskip.1cm

In Theorem \ref{Th1}, since the diffeomorphisms $\{B^m:\, m\in \N\}$ have uniformly bounded first derivatives and $\M$ is compact, 
the sequence  $\{B^m:\;m\in \N \}$ is uniformly bounded and equicontinuous.  Hence it has a subsequence $\{B^{m_\ell}\}$ that converges uniformly on $\M$ to a continuous map $\tilde B$. 
Then we choose a subsequence $\{m_{\ell_j}\}$ of $\{m_\ell\}$ such that $\{B^{-m_{\ell_j}}\}$ converges  uniformly on $\M$ to a continuous map $\bar B$. In Theorem \ref{Th2}, such a sequence is obtained using pre-compactness of $\{B^m:\, m\in \Z\}$ in $C^0(\M)$.
We have $\tilde B \circ \bar B=\bar B\circ \tilde B=\Id$, and so $\tilde B$ is a homeomorphism.
Thus there is a sequence $\{m_{\ell_j}\}$ such that $\{B^{m_{\ell_j}}\}$ converges to a homeomorphism $\tilde B$ in  $\dzm$. 
It follows that there is an increasing  sequence $\{n_k\}$ of the form $n_k=m_{\ell}-m_{\ell'}$ such that $\{B^{n_k}\}$ converges to the identity map on $\M$ in $\dzm$.

Since $\A_{p_0}=B$, we also have 
$$\A^{n_k}_{p_0}\to \Id \quad\text{in }\dzm.
$$
For this sequence $\{n_k\}$ the holonomies of $\A$ can be written as
$$
H_{x,{p_0}}^{s}=\underset{n\to\infty}{\lim}\,(\A_{p_0}^n)^{-1} \circ \A_x^n\,=\,\underset{k\to\infty}{\lim}\,\, (\A_{p_0}^{n_k})^{-1} \circ \A_x^{n_k}) \,=\, \underset{k\to\infty}{\lim}\, \A_x^{n_k}, 
$$
where convergence is in $\dzm$. Similarly, since ${p_0}=f^{-n_k}({p_0})$, we have
$$
H_{{p_0},x}^{u}\,=\,\underset{k\to\infty}{\lim}\, (\A_x^{-n_k})^{-1}\circ  \A_{p_0}^{-n_k} \,=\,
\underset{k\to\infty}{\lim}\, \A_{f^{-n_k}(x)}^{n_k}\circ \big( \A_{f^{-n_k}({p_0})}^{n_k}\big)^{-1}=\, \underset{k\to\infty}{\lim}\, \A_{f^{-n_k}(x)}^{n_k}. 
$$
Therefore,  with convergence  in $\dzm$, we have
\begin{equation}\label{CsCu}
(\Psi^s_x)^{-1}\circ \Psi^u_x \,=\, H^{s}_{x,{p_0}} \circ H^{u}_{{p_0},x}
=\underset{k\to\infty}{\lim}\, \A_x^{n_k} \circ \A_{f^{-n_k}(x)}^{n_k}=\, \underset{k\to\infty}{\lim}\, \A_{f^{-n_k}(x)}^{2n_k}.
\end{equation}
We will show that $\,\A_{f^{-n_k}(x)}^{2n_k}\to \Id$. 

Let $\delta'$ be as in  the Anosov Closing Lemma \ref{closing}.
We  consider an orbit segment $\{f^{-n_k}x, \dots,x, \dots f^{n_k}(x)\}$.
For all sufficiently large $k$ we have
$$\delta_k= \dist (f^{-n_k}x,f^{n_k}(x))\le K_x \la^{n_k} <\delta'.
$$ Hence there exists a periodic point $p_k=f^{2n_k}(p_k)$ whose orbit $\{p_k, \dots ,f^{2n_k}(p_k)\}$ 
is close to the orbit segment above, and in particular 
$$
\dist (f^{-n_k}x, p_k) \le c' \delta_k \quad \text{ and } \quad \dist ((f^{n_k}x), p_k) \le c' \delta_k. 
$$

Now we show  that $\A_{f^{-n_k}(x)}^{2n_k}$ is close to $\A_{p_k}^{2n_k}$. By the local product structure
there exits a unique point
$$y = W^s_{loc}(f^{-n_k}x)\cap W^u_{loc}(p_k).
$$ 
Then $ \dist (f^{-n_k}x,y) \le c_1 \delta_k$, and the orbit segment $\{y, \dots, f^{2n_k}(x)\}$ 
is close to both orbit segments $\{f^{-n_k}x,\dots, f^{n_k}(x)\}$ and $\{p_k, \dots, f^{2n_k}(p_k)=p_k\}$. 
It follows that  
$$
f^{2n_k}y = W^s_{loc}(f^{n_k}x)\cap W^u_{loc}(p_k),\; \text{ and so }\;\dist (f^{2n_k}y,p_k) \le c_1 \delta_k.
$$
Then applying \eqref{Anx} with $d_{C^0}$
to $f^{-n_k}x$ and $y$,\, and to $f^{2n_k}y$ and $f^{2n_k}p_k=p_k$ we obtain
$$
\begin{aligned}
&d_{C^0} ((\A_{y}^{2n_k})^{-1}\circ \A_{f^{-n_k}(x)}^{2n_k},\,\Id )  \,\le\, c\,\dist (y,f^{-n_k}(x))^{\beta(q-r)} \le c(c_1 \delta_k)^{\beta(q-r)}\quad\; \text{and}\\
&d_{C^0} (\A_{p_k}^{2n_k} \circ(\A_{y}^{2n_k})^{-1},\,\Id ) \, \le\, c(c_1 \delta_k)^{\beta(q-r)}.
\end{aligned}
$$
Now we show that $d_{0} (\A_{f^{-n_k}(x)}^{2n_k} ,\, \A_{p_k}^{2n_k} )\to 0$. We have 
$$
\begin{aligned}
&d_{0} (\A_{f^{-n_k}(x)}^{2n_k} ,\, \A_{p_k}^{2n_k} ) \le 
d_{0} (\A_{f^{-n_k}(x)}^{2n_k} ,\,\A_{y}^{2n_k})+ d_{C^0} (\A_{y}^{2n_k},\,\A_{p_k}^{2n_k})=\\
&d_{0} (\A_{y}^{2n_k}\circ \big (\A_{y}^{2n_k})^{-1}\circ\A_{f^{-n_k}(x)}^{2n_k},\, \A_{y}^{2n_k}\big) 
+d_{0} \big( \A_{y}^{2n_k} ,\, \A_{p_k}^{2n_k}\circ  (\A_{y}^{2n_k})^{-1}\circ \A_{y}^{2n_k} \big )= \\ &d_{0} (\A_{y}^{2n_k}\circ \big (\A_{y}^{2n_k})^{-1}\circ\A_{f^{-n_k}(x)}^{2n_k},\, \A_{y}^{2n_k}\big)  
+d_{0} \big( \Id,\, \A_{p_k}^{2n_k}\circ  (\A_{y}^{2n_k})^{-1}),  
\end{aligned}
$$
and the last term is at most $c(c_1 \delta_k)^{\beta(q-r)}\to 0$ as $k\to\infty$.
\vskip.1cm 
In Theorem \ref{Th1}, 
the iterates of $\A$ have uniformly bounded first derivatives, and so
$$
d_{0} (\A_{y}^{2n_k}\circ \big (\A_{y}^{2n_k})^{-1}\circ\A_{f^{-n_k}(x)}^{2n_k},\, \A_{y}^{2n_k}\big) 
\le c_2 \cdot d_{0} \big((\A_{y}^{2n_k})^{-1}\circ\A_{f^{-n_k}(x)}^{2n_k},\, \Id\big) \le c_4  \delta_k^{\beta(q-r)}.
$$
In Theorem \ref{Th2}, we use the assumption  $(2).$
The condition $\sigma^{2(1+ r)} \la^{\beta(q-r)}<1$ implies $\sigma^{2} \la^{\beta(q-r)}<1$, and 
we estimate 
$$
\begin{aligned}
& d_{0} (\A_{y}^{2n_k}\circ \big (\A_{y}^{2n_k})^{-1}\circ\A_{f^{-n_k}(x)}^{2n_k},\, \A_{y}^{2n_k}\big) 
\le  \|D \A_{y}^{2n_k}\|\cdot d_{0} (\big (\A_{y}^{2n_k})^{-1}\circ\A_{f^{-n_k}(x)}^{2n_k},\, \Id \big)\\
&\le  K \sigma^{2 n_k}\cdot  c(c_1 \delta_k)^{\beta(q-r)}
\le K \sigma^{2 n_k}\cdot  c_4 \la^{n_k \beta(q-r)}\le c_5 (\sigma^{2}\la^{\beta(q-r)})^{n_k} \to 0\, \text{ as }k\to\infty.
\end{aligned}
$$
Thus in both settings we have 
$$
d_{0} (\A_{f^{-n_k}(x)}^{2n_k} ,\, \A_{p_k}^{2n_k} ) \to 0 \quad\text{as $k\to\infty$.}
$$

We recall that $\A_{p_k}^{2n_k}=\Phi(p_k)^{-1} \circ B^{2n_k}\circ\Phi(p_k)$, and that $ B^{2n_k}\to \Id$ in $C^0$ since $ B^{n_k}\to \Id$ in $C^0$. We have 
$$
d_{0} (\A_{p_k}^{2n_k} , \, \Id) \,=\, d_{0}(\Phi(p_k)^{-1} \circ B^{2n_k}\circ\Phi(p_k),\,\Id) 
\,=\, d_{0} (\Phi(p_k)^{-1} \circ  B^{2n_k},\,\Phi(p_k)^{-1} ).
$$
In Theorem \ref{Th1}, since the set $\{\Phi(p)\}$ is bounded in $\donem$, this is at most 
$c_6\, d_{0}( B^{2n_k},\Id)$, which tends to zero.  In Theorem \ref{Th2}, the family of maps $\{\Phi(p)^{-1}\}$ is pre-compact in $C^0$ and hence equicontinuous, and it follows that 
$$d_{0} (\Phi(p_k)^{-1} \circ  B^{2n_k},\,\Phi(p_k)^{-1} )\to 0.
$$
Thus 
$$
d_{0} (\A_{p_k}^{2n_k} , \, \Id)\to 0\quad\text{and hence}\quad\A_{f^{-n_k}(x)}^{2n_k}\to \Id,
$$  
and using \eqref{CsCu} we conclude that 
$$
(\Psi^s_x)^{-1}\circ \Psi^u_x  \,=\, \underset{k\to\infty}{\lim} \, \A_{f^{-n_k}(x)}^{2n_k}= \Id, \quad\text{that is,}\quad \Psi^s_x=\Psi^u_x.
$$

We denote the set of homoclinic points for the fixed point ${p_0}$ by $S$.  For each $x\in S$ we define
$$
\Psi_x =\Psi^s_x=\Psi^u_x,
$$
which  gives a conjugacy between $\A$ and $\B$ on $S$. 
Since $\Psi_x=H^{s}_{{p_0},x}$, and the holonomies are in $\drm$, we have  $\Psi_x \in \drm$. 
\vskip.3cm

{\bf Extending the conjugacy $\Psi$ to $X$}  is done as in the proof of Theorem \ref{conjugate data}, but the estimates are simpler since the holonomies of $\B$ are trivial.



\vskip.5cm
\begin{thebibliography}{99}

\bibitem[AKL18]{AKL} A. Avila, A. Kocsard, and X. Liu.
{\em Liv\v{s}ic  theorem for diffeomorphism cocycles}.
Geometric and Functional Analysis 28 (2018), 943-964.

\bibitem[BK16]{BK} L. Backes and A. Kocsard. 
{\em Cohomology of dominated diffeomorphism-valued cocycles over hyperbolic systems.} 
Ergodic Theory  Dynam. Systems (2016), 36, 1703-1722.

\bibitem[Bo]{Bo} R. Bowen.  {\em Equilibrium States and the Ergodic Theory of Anosov Diffeomorphisms. }Lecture Notes in Mathematics, Vol. 470, Springer-Verlag, 1975.

\bibitem[dlLO99]{dlLO}  R. de la Llave and R. Obaya.
{\em Regularity of the composition operator in spaces of H\"older functions.}
Discrete and Continuous Dynamical Systems, 5 (1999), no. 1, 157-184.

\bibitem[dlLW10]{LW}  R. de la Llave and A. Windsor. 
              {\em Liv\v{s}ic theorem for non-commutative groups including groups  
              of diffeomorphisms, and invariant geometric structures.} 
               Ergodic  Theory   Dynam. Systems, 30, no. 4  (2010),  1055-1100.      

\bibitem[GG17]{GG} G. Grabarnik and M. Guysinsky.  {\it Liv\v{s}ic theorem for Banach rings.}
  Discrete and Continuous Dynamical Systems, Vol. 37, Issue 8 (2017), 4379-4390.
\vskip.05cm

\bibitem[K11]{K11} B. Kalinin. {\em Liv\v{s}ic theorem for matrix cocycles.} 
Annals of Mathematics, 173 (2011), no. 2, 1025-1042.

\bibitem[Liv71]{Liv1} A. N. Liv\v{s}ic. 
            {\em Homology properties of Y-systems.}  Math. Zametki 10, 758-763, 1971.

\bibitem[Liv72]{Liv2} A. N. Liv\v{s}ic.  {\em Cohomology of dynamical systems.} 
                      Math. USSR Izvestija 6, 1278-1301, 1972. 
      

\bibitem[KS26]{KS26} B. Kalinin and   V. Sadovskaya. {\it Periodic data rigidity for cocycles and hyperbolic automorphisms.} Preprint.

\bibitem[KtH]{KtH}  A. Katok and B. Hasselblatt. 
              {\em Introduction to the modern theory of dynamical systems.}  \\
              Encyclopedia of Mathematics and  its Applications
              {\bf 54}, Cambridge University Press,  1995.


\bibitem[KP16]{KP} A. Kocsard and R. Potrie.  {\em Liv\v{s}ic theorem for low-dimensional diffeomorphism cocycles}, Comment. Math. Helv. 91 (2016), 39-64.

              
              \bibitem[Pa99]{Pa99} W. Parry.
           {\em The Liv\v{s}ic periodic point theorem for non-Abelian cocycles.}
         Ergodic  Theory  Dynam. Systems, 19(3), 687-701, 1999.
              
\bibitem[S13]{S13}   V. Sadovskaya. {\em Cohomology of
               $GL(2,\R)$-valued cocycles over hyperbolic systems.}
               Discrete and Continuous Dynamical Systems, vol. 33, no. 5 (2013), 2085-2104.
              
 \bibitem[S15]{S15} V. Sadovskaya. {\em Cohomology of  fiber bunched cocycles over hyperbolic systems.}
Ergodic Theory  Dynam. Systems, Vol. 35, Issue 8 (2015), 2669-2688.     

\bibitem[S19]{S19}   V. Sadovskaya.  {\it Boundedness and invariant metrics for diffeomorphism cocycles over hyperbolic systems.} 
Geometriae Dedicata, Vol. 202, Issue 1 (2019), 401-417.  

\bibitem[S22]{S22}   V. Sadovskaya.   {\it Diffeomorphism cocycles over partially hyperbolic systems.} 
Ergodic Theory  Dynam. Systems.  Vol. 42, Issue 1 (2022), 263-286. 



\bibitem[Sch99]{Sch} K. Schmidt.
 {\em Remarks on Liv\v{s}ic theory for non-Abelian cocycles.}
      Ergodic Theory Dynam. Systems, 19(3), 703-721, 1999.


\end{thebibliography}
\end{document}